\documentclass[final,1p,times,twocolumn]{elsarticle}

\usepackage{amsmath, amssymb, amsthm}

\newtheorem{thm}{Theorem}[section]
\newtheorem{cor}[thm]{Corollary}
\newtheorem{lemma}[thm]{Lemma}
\newtheorem{prop}[thm]{Proposition}

\journal{}

\begin{document}

\begin{frontmatter}



\title{$p$-central action on groups}

 \author[label1]{Yassine Guerboussa}
 \address[label1]{University Kasdi Merbah Ouargla, Ouargla, Algeria \\ {\tt Email: yassine\_guer@hotmail.fr}}


\address{}

\begin{abstract}
Let $G$ be a finite $p$-group acted on faithfully by a group $A$.  We prove that if $A$ fixes every element of order dividing $p$ ($4$ if $p=2$) in a specified subgroup of $G$, then both $A$ and $[G,A]$ behave regularly, that is the elements of order dividing any power $p^i$ in each one of them form a subgroup; moreover $A$ and $[G,A]$ have the same exponent, and they are nilpotent of class bounded in terms of $p$ and the exponent of $A$.  This leads in particular to a solution of a problem posed by Y. Berkovich.  In another direction we discuss some aspects of the influence of  a $p$-group $P$ on the structure of a finite group which contains $P$ as a Sylow subgroup, under assumptions like  every element of order $p$ ($4$ if $p=2$) in a given term of the lower central series of $P$ lies in the center of $P$.

\end{abstract}


\end{frontmatter}


\section{Introduction}
Let $G$ be a finite group acted on by a group $A$.  It is convenient to say that $A$ acts $p$-centrally on $G$ if $A$ fixes every element of order dividing $p$ (4 if $p=2$) in $G$.  

For a positive integer $k$, the left normed commutator $[x_1,x_2,\dots,x_k]$ in $k$ elements of an ambient group,  can be defined by induction, $[x_1]=x_1$ and $[x_1,x_2,\dots,x_k]=[[x_1,x_2,\dots,x_{k-1}],x_k]$.  

We define $\gamma_k(G,A)$ to be the subgroup of $G$ generated by all the left normed commutators $[x_1,x_2,\dots,x_n]$, $n \geq k$, where the $x_i$'s lie in $G \cup A$ in such a way that $ x_1 \in G$, and at least $k-1$ of them lie in $A$.  Note that if one takes the natural action of $G$ on itself, $\gamma_k(G,G)$ coincides with $\gamma_k(G)$ the $k$th term of the lower central series of $G$.  Moreover we have $\gamma_k(G,A)$ is an $A$-invariant normal subgroup of $G$; this fact will be used freely below.\\ 

In \cite{Isa}, M. Isaacs proved that if $A$ is cyclic and  acts $p$-centrally on $[G,A]$, for all the primes $p$ dividing $|G|$, then there is a severe restriction on the structure of $[G,A]$ in terms of $n$ the order of $A$; for instance $[G,A]$ is nilpotent of class bounded by $n$, and has exponent dividing $n$.  The first purpose of this paper is to show in one hand that an analogue of Isaacs' result holds under the weaker condition that $A$ is a group of automorphisms of $G$ acting $p$-centrally on $ \gamma_p(G,A)$, and on the other hand to show that such a severe restriction applies also on $A$.

\begin{thm}\label{main1}
Let $G$ be a finite $p$-group and $A$ be a group of automorphisms of $G$,  such that  $A$ acts $p$-centrally on $ \gamma_p(G,A)$.  Then for all positive integer $i$,
\begin{description}
\item[(i)]the elements of $[G,A]$ of order dividing $p^i$ form a subgroup;
\item[(ii)]the elements of $A$ of order dividing $p^i$ form a subgroup;
\item[(iii)]$\operatorname{exp}([G,A]) =\operatorname{exp}(A)$;
\item[(iv)]the nilpotency class of both $A$ and $[G,A]$ does not exceed $n+p-2$, where $p^n=\operatorname{exp}(A)$.
\end{description}

\end{thm}   
This result is the best possible as shows the following example :

Let $E$ be an elementary abelian $p$-group of rank $p+1$, and let $A$ be the automorphism group of $E$ generated by the matrix 
$$ 
       \sigma  =  \left( 
       \begin{array}{cccc} 
       1   &  1   &     &             \\ 
           &  1       & \ddots    &     \\ 
           &      & \ddots     &  1   \\ 
          &      &     &  1   \\ 
       \end{array} 
       \right) 
       $$ 
Clearly, $\gamma_p(E,A)=[E,_pA]$, and $A$ acts $p$-centrally on it.  We have $\operatorname{exp}[E,A]=p$,  however it is easy to see that for any positive integer $n$, we have (with the convention $\binom{n}{i}=0$ for $i \geq n+1$)  
$$ 
       \sigma^n  =  \left( 
       \begin{array}{ccccc} 
       1   &  \binom{n}{1}   &\binom{n}{2}     & \cdots &   \binom{n}{p}   \\ 
           &  1       & \binom{n}{1}    & \ddots & \vdots     \\ 
           &      & \ddots     &  \ddots &\binom{n}{2}    \\ 
           &      &     &  1   & \binom{n}{1}\\ 
           &       &     &      &  1     \\
       \end{array} 
       \right) 
       $$
       
Therefore $\sigma^p \neq 1$, that is $\operatorname{exp}(A) \geq p^2$.\\

The following is an immediate consequence of Theorem \ref{main1}.  
 \begin{cor}\label{cor}
 Let $G$ be a finite $p$-group which acts $p$-centrally on $\gamma_p(G)$.  Then $G'$ and $G/\operatorname{Z}(G)$ have the same exponent.  
 \end{cor}
 Note that a particular version of Corollary \ref{cor} was proved by T. Laffey in \cite{Laf2}, where he established it under the condition  $\Omega(G) \leq \operatorname{Z}(G)$.  M. Y. Xu generalized Laffey's result to the case where $G$ acts $p$-centrally on $\gamma_{p-1}(G)$, with $p$ odd (see \cite[Corollary 4]{Xu1}).  The result is also known for a special class of $p$-groups of class $\leq p$, more precisely for $p$-groups of maximal class and order $\leq p^{p+1}$.  Moreover our proof implies, in fact, that  $[\Omega_k(G/\operatorname{Z}(G)),G] = \Omega_k(G')$, for any positive integer $k$.\\

As another application of Theorem \ref{main1}, we solve the following problem posed by Y. Berkovich.\\

{\bfseries Problem 1891 \cite{Berk3}}.  Do there exist a prime $p$ and a group $G$ of order $p^p$ and exponent $p$ such that
$p^2$ divides $\operatorname{exp}(\operatorname{Aut}(G))$?\\

Let $G$ be finite group of order $\leq p^p$ and exponent $p$, and let $A$ be a $p$-Sylow of $\operatorname{Aut}(G)$.  According to Lemma \ref{L0} (iii) below,  $\gamma_{i+1}(G,A) < \gamma_{i}(G,A)$, unless $\gamma_{i}(G,A)=1$.  Therefore $|\gamma_{i}(G,A)| \leq p^{p-i+1}$, so that $|\gamma_{p}(G,A)| \leq p$.  It follows that $A$ acts $p$-centrally on $\gamma_{p}(G,A)$, hence Theorem \ref{main1} yields   
\begin{cor} \label{Ya1}
Let $G$ be a non-cyclic group of order $\leq p^p$ and exponent $p$. Then  a $p$-Sylow of $\operatorname{Aut}(G)$ has exponent $p$.
\end{cor}

A finite group $G$ which acts (by conjugation) $p$-centrally on itself is termed $p$-central (this terminology is due to A. Mann, see \cite{Gon}).  The $p$-central $p$-groups  have many nice properties which qualify them to be dual to the powerful $p$-groups, we refer the reader to the introduction of \cite{Gon} for some basic facts on $p$-central $p$-groups.

J. Gonz\'alez-S\'anchez and T. Weigel introduced in \cite{Gon} a  class of groups that are more general than the $p$-central ones.  They called a group $G$ $p$-central of height $k$, if every element of order $p$ lies in the $k$th term of the upper central series of $G$.  The first main result in their paper is

\begin{thm}[Gonz\'alez-S\'anchez and T. Weigel]\label{GW}
 Let $G$ be a finite $p$-central group of height $k \geq 1$, with $p$ odd.  Then $G$ has a normal $p$-complement.
 \end{thm}          
A natural variant of a $p$-central group of height $k$, is a group that acts $p$-centrally on the $k$th term of its lower central series.  For $k=1$ and $p$ odd, the two definitions coincide.   A remarkable work in this context, which is not followed up extensively, was done by  Ming-Yao Xu in \cite{Xu2}.  He proved that a finite $p$-group ($p$ odd) satisfying $\Omega_1(\gamma_{p-1}(G)) \leq \operatorname{Z}(G)$ should behave regularly: the exponent of $\Omega_n(G)$ does exceed $p^n$, and moreover $|G:G^{p^n}| \leq |\Omega_n(G)|$, for all positive integer $n$.

The next result deals with the analogue of Theorem \ref{GW} for this dual class.  Note that this generalizes Lemma B in \cite{Isa}, with only a slight more effort.    The proof follows easily from Frobenius' normal $p$-complement theorem (see \cite[Theorem 4.5, p 253]{Gor}).      
 \begin{thm}\label{main2}
 Let $G$ be a finite group which acts $p$-centrally on $\gamma_i(G)$ for some positive integer $i$.  Then $G$ has a normal $p$-complement.
 \end{thm}

Let us note that the original proof of Theorem \ref{GW}  is based on Quillen stratification (see \cite[Theorem 3.1]{Gon}), as well as Quillen's $p$-nilpotency criterion (see \cite[Theorem 3.3]{Gon}).  We will give below a more elementary proof of it, which is based on Theorem \ref{main2}, and hence on the classic Frobenius' normal $p$-complement theorem.  Note also that our proof covers the prime $2$, however in that case we have to assume that $G$ is $4$-central of height $k \geq 1$.\\

Now we turn our attention to $p$-soluble groups.  Assume that $G$ is a finite $p$-soluble group, and that a $p$-Sylow of $G$ is $p$-central ($4$-central if $p=2$) of height $k$.  In \cite{Gon} it is proved that if $k \leq p-2$ and $p \neq 2$, then the $p$-length of $G$ is $\leq 1$. In a subsequent paper E. Khukhro (see \cite{Khu}) generalized this result and  showed that the $p$-length of $G$ is bounded above by $2m+1$, where $m$ is the largest integer satisfying $p^m-p^{m-1} \leq k$. By means of a theorem of P. Hall and G. Higman (see \cite[Theorem A (ii)]{Hall}), such a result holds if one can find an appropriate bound on  the exponent of a $p$-Sylow of $G/\operatorname{O}_{p'p}(G)$.  We have the following analogues of Khukhro's results.
\begin{thm}\label{main4}
Let $G$ be a finite $p$-soluble group such that $\operatorname{O}_{p'}(G)=1$.  If a $p$-Sylow $P$ of $G$ acts $p$-centrally on $\gamma_k(P)$, for some positive integer $k$; then the exponent of a $p$-Sylow of $G/\operatorname{O}_{p}(G)$ does not exceed $p^m$, where $m$ is the largest integer satisfying $p^m-p^{m-1} \leq k$.
\end{thm} 

This corollary follows easily from \cite[Theorem A (ii)]{Hall} for $p$ odd, and from \cite{Bry} for $p=2$.  
\begin{cor}
Let $G$ be a finite $p$-soluble group such that a $p$-Sylow $P$ of $G$ acts $p$-centrally on $\gamma_k(P)$, for some positive integer $k$.  Then the $p$-length of $G$ is bounded above by $2m+1$, where $m$ is the largest integer satisfying $p^m-p^{m-1} \leq k$.
\end{cor}    

The notation in this paper is standard.  Note only that $\Omega(G)$ stands for $\Omega_1(G)$ if $p$ is odd, and $\Omega_2(G)$ if $p=2$; and  $\Omega_{\{i\}}(G)$ denotes the set of all elements of $G$ having order dividing $p^i$.  

The basic results on regular $p$-groups can be found in \cite[Kap III]{Hup}, as well as in \cite{Berk1}.  We shall use them freely in the paper.

The remainder of the paper is divided into two sections.  Section 2 is devoted to proving Theorem \ref{main1}, and Section 3 to proving the remaining theorems stated in the introduction.

\section{\bf {\bf \em{\bf $p$-central action on $p$-groups}}}
\vskip 0.4 true cm
First, we collect some basic facts about the series $\gamma_i(G,A)$ defined in the introduction.
\begin{lemma}\label{L0}
Let $G$ be a finite group and $A$ be a group acting on $G$.  Then we have
\begin{enumerate}
\item $[\gamma_i(G,A),\gamma_j(A)] \leq \gamma_{i+j}(G,A)$, for $i,j \geq 1$;

\item   $\gamma_{i+1}(G,A)=\langle [\gamma_i(G,A),A,_nG], n \geq 0 \rangle$; 

\item If $A$ and $G$ are finite $p$-groups, then $\gamma_{i+1}(G,A) < \gamma_{i}(G,A)$, unless $\gamma_{i}(G,A)=1$.
\end{enumerate}
\end{lemma}
\begin{proof}

1. We proceed by induction on $j$.  Assume that $j=1$, if $c$ is a generator of $\gamma_i(G,A)$ and $a \in A$, then clearly $[c,a] \in \gamma_{i+1}(G,A)$.  Since $\gamma_{i+1}(G,A)$ is normal in $G$ and $A$-invariant, the claim follows for $j=1$.  Assume now that the result holds for $j$, it follows that $[\gamma_i(G,A),\gamma_j(A),A] \leq [\gamma_{i+j}(G,A),A] \leq \gamma_{i+j+1}(G,A) $, and $[A,\gamma_i(G,A),\gamma_j(A)] \leq [\gamma_{i+1}(G,A),\gamma_j(A)] \leq \gamma_{i+j+1}(G,A) $; the Three Subgroups Lemma yields $[\gamma_{j+1}(A), \gamma_i(G,A)] \leq  \gamma_{i+j+1}(G,A) $.  \\
2. It follows from the first property that all the subgroups   $[\gamma_i(G,A),A,_nG]$ lie in $\gamma_{i+1}(G,A)$.  Conversely let $c = [x_1,x_2,\dots,x_k]$ be a generator of $\gamma_{i+1}(G,A)$, and let be $s = \sup \{i \,|\, x_i \in A \} $.  We have  $ s \geq i+1$, so $[x_1,x_2,\dots,x_{s-1}] \in \gamma_i(G,A) $, hence $c \in [\gamma_i(G,A),A,_{k-s}G]$.\\ 

3.  Assume for a contradiction that $\gamma_{i+1}(G,A) = \gamma_{i}(G,A) \ne 1$. Let be $N = \gamma_{i}(G,A) \cap \operatorname{Z}(GA)$; we have $\gamma_{i}(G/N,A)=\gamma_{i}(G,A)/N$.  By induction on the order of $G$, if $\gamma_{i}(G,A) \neq N$ then $\gamma_{i+1}(G,A)N < \gamma_{i}(G,A)$ which contradicts our assumption.  Thus we have $\gamma_{i}(G,A) = N$.  It follows that $[\gamma_i(G,A),A,_nG]=1$ for all $n\geq 0$, so by the second property we have $\gamma_{i+1}(G,A) =1$, a contradiction.                     
\end{proof}
In the same spirit, the Three Subgroups Lemma yields
\begin{lemma}\label{L1}
Let $G$ be a group and $A$ be a group acting on $G$.  Let $k$ be an integer $\geq 2$, $H=[G,A]$, and  assume that  $A$ acts $p$-centrally on  $ \gamma_k(G,A)$.  Then $\Omega(\gamma_{k-1}(H)) \leq \Omega( \gamma_k(G,A)) \leq  \operatorname{Z}(H)$.    
\end{lemma}

\begin{proof}

We have $[\Omega( \gamma_k(G,A)),G,A]=[A,\Omega(\gamma_k(G,A)),G]=1$.   The Three Subgroups Lemma yields $[H,\Omega(\gamma_k(G,A))]=1$.  Since $\gamma_k(G,A) \leq \gamma_2(G,A)=H$, it follows that  $\Omega(\gamma_k(G,A)) \leq  \operatorname{Z}(H)$. 

Now we claim that $\gamma_{k-1}(H) \leq \gamma_k(G,A)$.  For $k=2$ this is trivial.  Assume that this is proved for $k$, and put $K=\gamma_k(G,A)$.  We have $[K,A,G]  \leq \gamma_{k+1}(G,A)$, and $[G,K,A] \leq [K,A] \leq \gamma_{k+1}(G,A)$.  It follows again from the Three Subgroups Lemma that $[H,K] \leq \gamma_{k+1}(G,A)$.  By assumption $\gamma_{k-1}(H) \leq K$, therefore $\gamma_{k}(H) \leq \gamma_{k+1}(G,A)$.
\end{proof}
In \cite{Xu1} Ming-Yao Xu showed that if a finite $p$-group $G$, $p$ odd, satisfies $\Omega_1(\gamma_{p-1}(G)) \leq \operatorname{Z}(G)$, then $G$ is strongly semi- $p$-abelian, in other words $G$ satisfies the property :
$$(xy^{-1})^{p^n}=1 \Leftrightarrow x^{p^n}=y^{p^n} \mbox{ for any postive integer } n.$$
 Such a  group shares many properties with the regular ones (see \cite{Xu2}).  For instance the exponent of $\Omega_n(G)$ does exceed $p^n$, and $|G:G^{p^n}| \leq |\Omega_n(G)|$. The above does not hold for $p=2$ as shows the quaternion group.  However if one requires that $\Omega_2(G) \leq \operatorname{Z}(G)$ we cover the case of  $2$-central $2$-groups which is covered for instance in \cite[Corollary 2.2]{Isa}.  Hence we have
\begin{lemma}[Ming-Yao Xu]\label{L4}
Let $G$ be a finite $p$-group such that $\Omega(\gamma_{p-1}(G)) \leq \operatorname{Z}(G)$.  Then for any positive integer $n$, we have 
\begin{description}
\item[(i)] the elements of $G$ of order dividing $p^n$ form a subgroup;
\item[(ii)] for $p$ odd, $|G:G^{p^n}| \leq |\Omega_n(G)|$. 
\end{description}
  
\end{lemma}

The above result combined with Lemma \ref{L1} yields

\begin{cor}\label{C5}
Let $G$ be a finite $p$-group and $A$ be  a group  acting on $G$, such that $A$ acts $p$-centrally on $\gamma_p(G,A)$. Then for any positive integer $n$, and for $H=[G,A]$, we have 
\begin{description}
\item[(i)] $\operatorname{exp}(\Omega_n(H)) \leq p^n$;
\item[(ii)] for $p$ odd, $|H:H^{p^n}| \leq |\Omega_n(H)|$. 
\end{description}

\end{cor}    
The following  lemma generalizes a result of I. M. Isaacs (see \cite [Theorem 2.1]{Isa}).     
 \begin{lemma}\label{L2}
Let $G$ be a finite $p$-group and $A$ be  a group of order $p$ acting on $G$, such that $A$ acts $p$-centrally on $\gamma_p(G,A)$.  Then $[G,A]$ has exponent at most $p$.     
\end{lemma}

\begin{proof} 
Assume first that $p$ is odd.  By induction on $|G|$ we may assume that the result holds for any smaller $p$-group.
We have $[G,A] < G$, thus  by induction $[G,A,A]$ has at most exponent $p$.  This yields $[G,A,A] \leq \Omega_1(H)$; since $\Omega_1(H)$ is normal in $G$, Lemma \ref{L0}(ii) and Corollary \ref{C5} (i) imply that $\gamma_3(G,A)$ has exponent $\leq p$.  Our condition on the action of $A$ on $G$ implies in particular that $[\gamma_p(G,A),A]=1$, and again Lemma \ref{L0}(ii) yields $\gamma_{p+1}(G,A)=1$.  

Now let  $K$ denote the semidirect product $ [G,A]A$.  We claim that $\gamma_p(K)=1$.  This follows at once if one proves that  $\gamma_{i-1}(K) \leq \gamma_{i}(G,A)$ for all $i \geq 3$.  Assume that $i=3$.  It follows easily from the Three Subgroup Lemma that $[[G,A],[G,A]] \leq \gamma_3(G,A)$, and clearly we have $[[G,A],A] \leq \gamma_3(G,A)$.  This amounts to saying that $[G,A]/\gamma_3(G,A)$ lies in the center of $K/\gamma_3(G,A)$, but since $K/[G,A] \cong A$ is cyclic, we have $K/\gamma_3(G,A)$ is abelian. Hence $\gamma_2(K) \leq \gamma_3(G,A)$.  Now by induction we may assume that $\gamma_{i-1}(K) \leq \gamma_i(G,A)$.  We have $[\gamma_i(G,A),A] \leq \gamma_{i+1}(G,A)$, and the Three Subgroups Lemma implies that $[\gamma_i(G,A),[G,A]]\leq \gamma_{i+1}(G,A)$.  Thus   $[\gamma_{i-1}(K),K] \leq [\gamma_i(G,A),K] \leq \gamma_{i+1}(G,A)$.

As $\gamma_p(K)=1$, it follows that $K$ is regular;  moreover as $K' \leq \gamma_{3}(G,A)$,  we have $\operatorname{exp}(K')\leq p$. Thus $(xy)^p=x^py^p$ for all $x,y \in K$; in other words $K$ is $p$-abelian.   

Let be $g \in G$, $\sigma \in A$ and set $x = [g,\sigma]$.  we have
$$g^{-1}g^{\sigma^p}=x^{1+\sigma+...+\sigma^{p-1}}=1$$

Since $K$ is $p$-abelian, we have
$$x^{1+\sigma+...+\sigma^{p-1}}=(x\sigma^{-1})^p \sigma^p = x^p \sigma^{-p} \sigma^p = x^p$$
This shows that $[G,A]$ is generated by elements of order not exceeding $p$; as $K$ is regular it follows that the exponent of $[G,A]$ is at most $p$.

For $p=2$, we may assume that $[G,A,A]$ has exponent $\leq 2$.  Under the above notation let be $y=[x,\sigma]$.  We have $1=g^{-1}g^{\sigma^2}=x^2y$ and  $y^2=1$, hence $x^4=1$.  As $A$ fixes every element of order $4$ in $[G,A]$, we have $y=[x,\sigma]=1$.  Thus $x^2=1$ and on the other hand $\gamma_{3}(G,A)=1$ . Now as $[[G,A],[G,A]] \leq \gamma_{3}(G,A)=1$,  it follows that $[G,A]$ is abelian , and it is generated by elements of order $\leq 2$, the result follows.  

\end{proof}

The following lemma is useful for induction.

 \begin{lemma}\label{L6}
Let $G$ be a finite $p$-group and $A$ be  a $p$-group acting on $G$. If $A$ acts $p$-centrally on $\gamma_{k}(G,A)$, for some $1 \leq k \leq p$, then the same holds if we replace $G$ by $G/ \Omega_i(H)$ for any positive integer $i$, where $H$ denotes $[G,A]$.  

\end{lemma}
\begin{proof}
Assume first that $i=1$ and  let us denote $G/ \Omega_1(H)$ by $\overline{G}$.  We have $\gamma_k(\overline{G},A)=\overline{\gamma_{k}(G,A)}$. Let $\overline{x} \in \gamma_k(\overline{G},A)$ be an element of order $\leq p$ ($\leq 4$ if $p=2$), so we can assume that $x \in \gamma_{k}(G,A)$. As $\gamma_{p}(G,A) \leq \gamma_{k}(G,A)$, it follows from Corollary \ref{C5} (i) that $x^p \in \Omega(\gamma_{k}(G,A))$. Therefore $[x^p,A]=1$, and by Lemma \ref{L1} $x^p$ lies in the center of $H$.  Thus $x$ induces  on  $K=HA$ (by conjugation) an automorphism of order $\leq p$.    

It follows easily that $\gamma_k(K,\langle x \rangle) \leq \gamma_{k-1}(H,\langle x \rangle) \leq \gamma_{k-1}(H) $.  By Lemma \ref{L1}, $\gamma_{k-1}(H) \leq \operatorname{Z}(H)$, thus the inner automorphism induced by $x$ on $K$ acts $p$-centrally on $\gamma_k(K,\langle x \rangle)$, so it acts $p$-centrally on $\gamma_{p}(K,\langle x \rangle)$ .  It follows at once from Lemma \ref{L2} that $[x,A]^p=1$, that is $[\overline{x},A]=\overline{1}$. 

Now by induction we may assume the result is true for $G/\Omega_i(H)$.  Since $[G/\Omega_i(H),A]=H/\Omega_i(H)$, Corollary \ref{C5} (i) implies that $\Omega_1([G/\Omega_i(H),A])= \Omega_{i+1}(H)/\Omega_i(H)$, so by the first step the result holds for $G/\Omega_{i+1}(H)$.        
\end{proof}
It follows from the lemma above that
\begin{lemma}\label{L'6}
Let $L$ denote $\gamma_{k}(G,A)$. Under the assumptions of Lemma \ref{L6}, we have $A$ acts $p$-centrally on $L/\Omega_i(L)$, and $[A,\Omega_{i}(L)] \leq \Omega_{i-1}(L)$, for all positive integer $i$.  
\end{lemma}

\begin{proof}
We have 
$\gamma_k(G/\Omega_i(H),A)=\gamma_k(G,A)\Omega_i(H)/\Omega_i(H)=L\Omega_i(H)/\Omega_i(H)$.  Also $L\Omega_i(H)/\Omega_i(H)$ and $L/L \cap \Omega_i(H)$ are canonically isomorphic as $A$-groups.   It follows from Corollay \ref{C5} (i) that $L \cap \Omega_i(H)=\Omega_i(L)$. Now Lemma \ref{L6} implies that $A$ acts $p$-centrally on $L/\Omega_i(L)$.   Also Corollay \ref{C5} (i) implies that $\Omega_1(L/\Omega_{i-1}(L))= \Omega_{i}(L)/\Omega_{i-1}(L)$, thus $[A,\Omega_{i}(L)] \leq \Omega_{i-1}(L)$.

\end{proof}
In Lemma \ref{L6} as well as Lemma \ref{L'6} we assumed that $A$ is a $p$-group. The following result shows that this assumption can be dropped if one assumes that $A$ acts faithfully on $G$. The result in a seemingly weaker form is classic (see \cite[Satz IV.5.12]{Hup}). 
 \begin{prop}\label{P7}
Let $G$ be a finite $p$-group and $A \leq \operatorname{Aut}(G)$. If $A$ acts $p$-centrally on $\gamma_{i}(G,A)$ for some $ i \geq 1$, then $A$ is a $p$-group. 
\end{prop}
\begin{proof}
Let $Q$ be a $q$-Sylow of $A$, $q \neq p$.  By \cite[Theorem 3.6, p 181]{Gor}, $[G,Q,Q]=[G,Q]$, so $[G,_iQ]=[G,Q]$.  As $[G,_iQ] \leq \gamma_{i+1}(G,Q)$, it follows that $Q$ acts $p$-centrally on $[G,Q]$, and from \cite[Lemma 4.1]{Isa} it follows that $Q$ is a $p$-group.  Thus $Q=1$.  
   
\end{proof}
The collection of the lemmas above yields the following key result.
\begin{prop}\label{P8}
Let $G$ be a finite $p$-group and $A$ be a group of automorphisms of $G$,  such that  $A$ acts $p$-centrally on $\gamma_{p}(G,A)$.  Then an element $\sigma \in A$ satisfies $\sigma^{p^n}=1$ if and only if $\sigma$ acts trivially on $G/ \Omega_n(H)$, where $H=[G,A]$. 
\end{prop}
\begin{proof}
For $n=1$, let be $\sigma \in A$ such that  $\sigma^p=1$. It follows immediately from Lemma \ref{L2} that  $[G,\sigma]$ has exponent $ \leq p$, that is $\sigma$ acts trivially on $G/ \Omega_1(H)$.   

Conversely, assume that  $ [g,\sigma] \in \Omega_1(H)$, for all $g \in G$.  Let be $K = [G,\langle \sigma \rangle] \langle \sigma \rangle$; as we seen in the proof of  Lemma \ref{L2} $\gamma_{p}(K) \leq \gamma_{p+1}(G,\langle \sigma \rangle)$, and as $\gamma_{p}(G,\langle \sigma \rangle)$ has exponent $p$, it follows that $\gamma_{p+1}(G,\langle \sigma \rangle)=1$.  Therefore $K$ has class $\leq p-1$, so it  is regular, and since $K' \leq [G,\langle \sigma \rangle]$ has exponent $p$, $K$ is $p$-abelian.  Now let be $g \in G$ and $x=[g, \sigma]$.  We have  
$$g^{-1}g^{\sigma^p}=x^{1+\sigma+...+\sigma^{p-1}}=(x\sigma^{-1})^p \sigma^p = x^p \sigma^{-p} \sigma^p=1$$
thus $\sigma$ has order $\leq p$.

Now we proceed by induction on $n$.  If $\sigma^{p^{n+1}}=1$, then $\sigma^p$ has order at most $p^n$, so by induction $g^{-1}\sigma^p(g)  \in \Omega_n(H)$ for all $g \in G$.  Therefore $\sigma$ acts on $G/ \Omega_n(H)$ as an automorphism of order $\leq p$. By Lemma \ref{L6}, $\sigma$ acts $p$-centrally on $\gamma_p(G/ \Omega_n(H), A)$,  therefore $[G/ \Omega_n(H), \langle \sigma \rangle]$ has exponent $\leq p$ by Lemma \ref{L2}.  Thus we have $[g, \sigma]^p \in \Omega_n(H)$ for all $g \in G$, and by Corollary \ref{C5} (i), $[g, \sigma]^{p^{n+1}}=1$ for all $g \in G$.     

Similarly, assume that  $[G,\sigma] \leq \Omega_{n+1}(H)$.  By Corollary \ref{C5} (i), $\operatorname{exp}(\Omega_{n+1}(H)) \leq p^{n+1}$, hence we have $[g,\sigma]^{p^{n+1}} =1$, for all $g \in G$.  Therefore the automorphism $\sigma'$ induced by $\sigma$ on $G/ \Omega_n(H)$ satisfies  $[g\Omega_n(H), \sigma']^p = 1$.  We deduce from the first step that $\sigma'$ has order $ \leq p$.  Thus $g^{-1} \sigma^p(g) \in \Omega_n(H)$ for all $g \in G$.  By induction the order of $\sigma^p$ is at most $p^n$, the result follows.

\end{proof}

Now we can prove our first main theorem.

\begin{proof}[Proof of Theorem \ref{main1}]
(i) this is Corollary \ref{C5} (i).

(ii) Proposition \ref{P8} implies that $\Omega_{\{i\}}(A) = \operatorname{C}_A(G/\Omega_i(H))$, thus $\Omega_{\{i\}}(A)$ is a subgroup of $A$.\\
(iii) Let be $\operatorname{exp}(H)=p^n$ and $\operatorname{exp}(A)=p^m$.  We have $\Omega_n(A)=\operatorname{C}_A(G/\Omega_n(H))=\operatorname{C}_A(G/H)=A$,  it follows from (ii) that $p^m \leq p^n$ .  Again we have $A=\Omega_m(A)=\operatorname{C}_A(G/\Omega_m(H))$.  Therefore $[G,A]=H=\Omega_m(H)$.  By Corollary \ref{C5} (i), $p^n \leq p^m$. \\
(iv) By Lemma \ref{L1}, $H$ acts $p$-centrally on $\gamma_{p-1}(H)$. Lemma \ref{L'6} yields $[\Omega_{i}(\gamma_{p-1}(H)),H] \leq \Omega_{i-1}(\gamma_{p-1}(H))$, for all $i \geq 1$.  Therefore 
$$1 \leq \Omega_{1}(\gamma_{p-1}(H)) \leq \Omega_{2}(\gamma_{p-1}(H)) \leq ...\leq \Omega_{n}(\gamma_{p-1}(H))=\gamma_{p-1}(H) \leq \gamma_{p-2}(H)\leq ...\leq H$$
is a central series of $H$.  This proves that $H$ is nilpotent of class $\leq  n+p-2$.

Similarly, let be $L=\gamma_p(G,A)$;  Lemma \ref{L'6} yields $[A,\Omega_{i}(L)] \leq \Omega_{i-1}(L)$, for all positive integer $i$.  Therefore $A$ stabilizes the normal series
$$1 \leq \Omega_{1}(L) \leq \Omega_{2}(L) \leq ...\leq \Omega_{n}(L)= \gamma_p(G,A) \leq \gamma_{p-1}(G,A)\leq ...\leq \gamma_2(G,A) \leq G$$             
It follows at once from the well known result of Kaloujnine (see \cite[Satz III.2.9]{Hup}) that $A$ is nilpotent of class $\leq  n+p-2$.
\end{proof}

\section{\bf {\bf \em{\bf $p$-central action, $p$-nilpotency and $p$-solubility length}}}
\vskip 0.4 true cm

In the following proofs we need only Proposition \ref{P7}, which is by the way independent from the material developed in the previous section.

 \begin{proof}[Proof of Theorem \ref{main2}]
  Let $P$ be a non-trival $p$-subgroup of $G$.  We have $A=\operatorname{N}_G(P)/\operatorname{C}_G(P)$ acts faithfully on $P$, and $p$-centrally on $\gamma_i(P,A)$.  By Proposition \ref{P7}, $A$ is a $p$-groups.  The result follows now from Frobenius' criterion of $p$-nilpotency (see \cite[Theorem 4.5, p 253]{Gor}).        
   
 \end{proof}

\begin{proof} [Elementary proof of Theorem \ref{GW}]
Let be $C= \operatorname{C_G(\Omega(G))}$.  As $\Omega(C) \leq \Omega(G)$, we have $\Omega(C) \leq \operatorname{Z}(C)$.  Therefore $C$ acts $p$-centrally on itself.  It follows from Theorem \ref{main2} that $C$ has a normal $p$-complement $N$, and since $N$ is characteristic in $C$, $N$ is also normal in $G$.  Now we have only to show that $N$ has a $p$-power index, indeed, As $\Omega(G) \leq \operatorname{Z}_k(G)$, it follows that $\Omega(G)$ is nilpotent, so it is a $p$-group.  On the other hand it follows by induction that $[\Omega(G),_iG] \leq \operatorname{Z}_{k-i}(G)$, for $i \leq k$.  In particular $[\Omega(G),_kG]=1$.  Therefore $A=G/C$ acts faithfully on $\Omega(G)$ and $p$-centrally on $\gamma_{k+1}(G,A)$.  From Proposition \ref{P7}, one deduces that $|G:C|$ is a $p$-power, and clearly $|C:N|$ is a $p$-power.     

\end{proof}

The remainder part is devoted to proving Theorem \ref{main4}.  The proof is inspired by that of Khukhro in \cite{Khu}. Let us quote some celebrated results that we need in the sequel.

Let $P$ be a finite $p$-group.  Recall that a Thompson critical subgroup of $P$ is a characteristic subgroup $C$ having the following properties:
\begin{enumerate}
\item $[C,P] \leq \operatorname{Z}(C)$;
\item $C/\operatorname{Z}(C)$ is elementary abelian;
\item $C$ is self centralizing, that is $\operatorname{C}_G(C)=\operatorname{Z}(C)$;
\item any non trivial $p'$-automorphism of $P$ acts non trivially on $C$.
\end{enumerate}

A celebrated result of Thompson asserts that every finite $p$-group $P$ has a Thompson critical subgroup $C$ (see\cite[Theorem 5.3.11]{Gor}). By Proposition \ref{P7}, any non trivial $p'$-automorphism of $P$ acts non trivially on $K=\Omega(C)$, and  the first two properties of $C$ imply that $K$ has exponent $p$ ($\leq 4$ if $p=2$) . Hence we have the following well-known result.

\begin{thm}[Thompson] \label{Tho}
Let $P$ be a finite $p$-group.  Then $P$ has a subgroup $K$ of exponent $p$ ($4$ if $p=2$) and class at most $2$, such that any non trivial $p'$-automorphism of $P$ acts non trivially on $K$.

\end{thm} 

We need also the following weaker form of the Hall-Higman theorem (see \cite[Theorem B]{Hall}).
\begin{thm}[Hall-Higman] \label{HH}
Let $V$ be a vector space over a field of characteristic $p$, and $G$ be a $p$- soluble group of automorphisms of $V$ such that $\operatorname{O}_p(G)=1$.  If $g$ is an element of $G$ of order $p^m$, then the minimal polynomial of $g$ is $(X-1)^r$, where $r$ satisfies $p^m-p^{m-1}\leq r \leq p^m$.
\end{thm} 

\begin{proof}[Proof of Theorem \ref{main4}]
Let $p^n$ be the exponent of a $p$-Sylow of $G/\operatorname{O}_p(G)=\overline{G}$, and $Q$ be a $p$-complement in $\operatorname{O}_{pp'}(G)$.\\

{\bfseries Claim 1.} There is $g \in G$ which normalizes $Q$ such that $\overline{g}=g\operatorname{O}_p(G)$ has order $p^n$.

Indeed, Let be $y \in G$ such that $\overline{y}$ has order $p^n$.  As any two $p$-complements in $\operatorname{O}_{pp'}(G)$ are conjugate, $Q$ and $Q^y$ are conjugate in $\operatorname{O}_{pp'}(G)$.  Moreover since $\operatorname{O}_{pp'}(G)= \operatorname{O}_{p}(G)Q$, there is $x \in \operatorname{O}_{p}(G)$ such that $Q^y=Q^x$.  Now it suffices to take $g=yx^{-1}$.\\         

{\bfseries Claim 2.} $\overline{g}$ acts faithfully on $\overline{Q}=Q\operatorname{O}_{p}(G)/\operatorname{O}_{p}(G)$, so that $[Q,{\overline{g}}^{p^{n-1}}] \neq 1$.

Indeed, as $\overline{Q}=\operatorname{O}_{p'}(\overline{G})$, it follows from \cite[Theorem 6.3.2]{Gor} that $\overline{Q}$ is self centralizing.  Thus $\langle \overline{g} \rangle \cap \overline{Q} = \overline{1}$.\\

Now let  $K$ be a subgroup of $\operatorname{O}_{p}(G)$ as in Theorem \ref{Tho}.   Consider a series of $\operatorname{O}_{p}(G)$-invariant subgroups 
$$ 1 \leq K_1 \leq ...\leq K_n=K$$
with elementary abelian sections, on which $\operatorname{O}_{p}(G)$ acts trivially.  Whence there is a well defined action (induced by conjugation) of the semi-direct product   $A=Q\langle \overline{g} \rangle$ on each section $K_{i+1}/K_i$.  Since $\operatorname{O}_{p'}(G)=1$, $\operatorname{O}_{p}(G)$ is self centralizing by \cite[Theorem 6.3.2]{Gor}; therefore    $[Q,g^{p^{n-1}}]$ is a non trivial $p'$-group of automorphisms of $\operatorname{O}_{p}(G)$, and it acts non trivially on $K$ by Theorem \ref{Tho}. Therefore $[Q,g^{p^{n-1}}]$ acts non trivially on some section $V=K_{i+1}/K_i$.\\    

{\bfseries Claim 3.} $\operatorname{O}_{p}(A)=1$, and $\overline{g}$ acts faithfully on $V$.\\
   
Assume first that Claim 3 is true.  It follows from Theorem \ref{HH} that $(\overline{g}-\operatorname{1}_V)^s \neq 0$, where $s=p^n-p^{n-1}-1$.  Thus for some $x \in K$, $[x,_sg] \neq 1$.  On the other hand any $p$-Sylow $P$ of $G$ acts $p$-centrally on $\gamma_k(P)$, and since $K$ has exponent $p$ ($4$ if $p=2$) we have $[x,_kg]=1$.  This implies that $k \geq s+1 = p^n-p^{n-1}$.

It remains to prove Claim 3.  As $\langle \overline{g} \rangle$ is a $p$-Sylow in $A$, if $\operatorname{O}_{p}(A) \neq 1$, then $\overline{g}^{p^{n-1}} \in \operatorname{O}_{p}(A)$.  Hence $1 \neq [Q,\overline{g}^{p^{n-1}}] \leq \operatorname{O}_{p}(A) \cap \operatorname{O}_{p'}(A)=1$,  a contradiction.  Also by definition of $V$, $[Q,\overline{g}^{p^{n-1}}]$ acts non trivially on it, so $\overline{g}^{p^{n-1}}$ acts non trivially on $V$.  Thus $\overline{g}$ acts faithfully on $V$.

\end{proof}

\begin{center}{\textbf{Acknowledgments}}

A. Mann taked my attention to $p$-central $p$-groups, and provided me with a copy of \cite{Gon}. An early discussion with M. I. Isaacs about his results in \cite{Isa} was very helpful,  and T. Laffey provided me with some of his very old papers. I'm really grateful to all of them. 
\end{center}

\vskip 0.4 true cm




\bibliographystyle{model1a-num-names}
\bibliography{<your-bib-database>}



\end{document}